\documentclass[12pt,oneside]{article}
\usepackage{amsfonts}
\usepackage{amsmath}
\usepackage{amssymb}
\usepackage{amsthm}
\usepackage{enumerate}
\usepackage[latin1]{inputenc} 
\usepackage{color}
\usepackage{mathtools}
\usepackage{pifont}
\usepackage{hyperref}
\usepackage{graphicx}
\usepackage[table]{xcolor}
\graphicspath{ {images/} }

\newtheorem{theorem}{Theorem}%[section]

\newtheorem{lemma}[theorem]{Lemma}
\theoremstyle{definition}
\newtheorem{example}[theorem]{Example}

\newtheorem{definition}[theorem]{Definition}

\newcommand\y{\cellcolor{lightgray}}

\newcommand\z{\multicolumn{1}{|r}}

\newenvironment{smat}{\left[\begin{smallmatrix}}{\end{smallmatrix}\right]}

\usepackage[paperwidth=210mm,paperheight=297mm,textwidth=180mm,textheight=255mm,top=15mm,lmargin=15mm]{geometry}

\title{Companion unit lower Hessenberg matrices
}
\author{Alberto Borobia and Roberto Canogar\\
\small Dpto. Matem\'aticas, Universidad Nacional de Educaci\'on a Distancia (UNED), Madrid, Spain\\
\small e-mail: $aborobia@mat.uned.es, \ rcanogar@mat.uned.es$ }
\date{\today}

\begin{document}

\maketitle

\begin{abstract}

In recent years there has been a growing interest in companion matrices. There is a deep knowledge of sparse companion matrices, in particular it is known that every sparse companion matrix can be transformed into a  unit lower Hessenberg matrix of a particularly simple type by any combination of transposition, permutation similarity and diagonal similarity. The latter is not true for the companion matrices that are non-sparse, although it is known that every non-sparse companion matrix is nonderogatory. In this work the non-sparse companion matrices that are unit lower Hessenberg will be described. A natural generalization  is also considered.

\end{abstract}

\section{Introduction}

For a given field $\mathbb{F}$, a matrix $A$ of order $n$ is said to be \emph{companion} as long as: $(i)$ $A$ has $n^2 - n$  entries that are constants of $\mathbb{F}$; $(ii)$ the $n$ remaining entries of $A$ are the variables $x_1,\ldots,x_n$; and $(iii)$   the characteristic polynomial of $A$ is
\begin{equation}\label{PolynomialOfCompanion}
\det(\lambda I_n-A)=\lambda^n-x_1\lambda^{n-1}-\cdots-x_{n-1}\lambda-x_n.
\end{equation}
The classical example is the Frobenius companion matrix
$$\begin{bmatrix}
0 & 1 & 0 & \cdots & 0   \\
0 & 0 & 1 & \ddots & 0   \\
\vdots &  & \ddots & \ddots & 0   \\
0 & 0 & \cdots & 0 & 1   \\
x_n & x_{n-1} & \cdots & x_2 & x_1   \\
\end{bmatrix}.$$ 
In 2003 Fiedler~\cite{Fiedler} introduced new companion matrices, which in turn  produced an increased interest in companion  matrices. In 2013 Ma and Zhan~\cite{Ma-Zhan} studied those matrices of order $n$ whose entries are in the field $\mathbb{F}(x_1,\ldots,x_n)$ of rational functions in the variables $x_1,\ldots,x_n$ and for which the characteristic polynomial is~(\ref{PolynomialOfCompanion}). They showed that such matrices are necessarily irreducible and have at least $2n-1$ nonzero entries. This  permits to define a \emph{sparse companion matrix} as a companion matrix with $2n-1$ nonzero entries. In 2016 Garnett et al.~\cite{Garnett-Shader-Shader-VanDerDiessche}  characterized the Ma-Zhan matrices. 

A square matrix is a \emph{unit lower Hessenberg matrix}, \emph{ULH matrix} by short, if all its  superdiagonal entries are equal to one and all its entries above the superdiagonal are equal to zero. In 2014 Eastmann  et al.~\cite{Eastman-Kim-Shader-VanderMeulen,Eastman-Kim-Shader-VanderMeulen2}  showed that every sparse companion matrix can  be transformed into a ULH matrix by any combination of transposition, permutation similarity and diagonal similarity (we will be more precise in Theorems~\ref{CharSparseCompanion} and~\ref{HMinusCNoCompanion} below). 
In 2019 Deaett et al.~\cite{Deaett-Fischer-Garnett-VanderMeulen} gave  one example that shows  that this is not necessarily the case for non-sparse companion matrices. On the other hand, they  proved that any companion matrix is nonderogatory. 

If $A$ is a companion matrix then $(1,-x_1,\ldots,-x_n)$ are the coordinates of $\det(\lambda I_n-A)$ with respect to the monomial basis $\{\lambda^n, \lambda^{n-1}, \ldots, \lambda,1\}$ of   the $(n+ 1)$-dimensional vector space $\mathbb{F}_n[\lambda]$ of polynomials of degree at most $n$ on the variable $\lambda$. The concept of companion matrix can be generalized whenever  we consider any polynomial basis of    $\mathbb{F}_n[\lambda]$  other than the monomial basis.
For example, consider the Newton basis 
\begin{align*} %\label{NewtonBasis}
\left\{\prod_{h=1}^{n}(\lambda-\gamma_h),\;\; \prod_{h=1}^{n-1}(\lambda-\gamma_h),\;\; \ldots,\;\; (\lambda-\gamma_1), 1\right\}    
\end{align*} 
that occurs on interpolation at the points $\lambda=\gamma_1, \ldots, \lambda=\gamma_{n}$. 
Actually, Farahat and Ledermann~\cite{Farahat-Ledermann} noted that a Frobenius companion matrix in which we modify the main diagonal,
\begin{equation*} \label{FL}
    A=\begin{bmatrix}
\gamma_1 & 1 & 0 & \cdots & 0 \\
0 & \gamma_2 & 1 & \ddots & \vdots \\
\vdots & \ddots & \ddots & \ddots & 0 \\
0 & \cdots & 0  & \gamma_{n-1} & 1 \\
x_n & \cdots & \cdots & x_2 & x_1+\gamma_n 
\end{bmatrix},
\end{equation*}
has characteristic polynomial 
$$\det(\lambda I_n-A)=\prod_{h=1}^{n}(\lambda-\gamma_h)-x_1\prod_{h=1}^{n-1}(\lambda-\gamma_h)-\cdots-x_{n-1}(\lambda-\gamma_1)-x_n.$$
Therefore $(1,-x_1,\ldots,-x_n)$ are the coordinates of $\det(\lambda I_n-A)$ with respect to the Newton basis. The entry $x_1+\gamma_n$ causes matrix $A$ to not  meet  the condition $(ii)$ of the definition of  companion matrix. Nevertheless this can be sidestepped by making $\gamma_n=0$.

By results of~\cite{Barnett75, Barnett81, Maroulas-Barnett}, to name a few references in this direction, we can go even further and 
regard 
\begin{equation} \label{Barnet}
\begin{bmatrix}
d_{11} & 1 & 0 & \cdots & 0 \\
d_{21} & d_{22} & 1 & \ddots & \vdots \\
\vdots & \ddots & \ddots & \ddots & 0 \\
d_{n-1, 1} & \cdots & d_{n-1, n-2}  & d_{n-1, n-1} & 1 \\
x_n & \cdots & x_3 & x_2 & x_1 
\end{bmatrix},
\end{equation}
as a  companion matrix with respect to some polynomial basis of $\mathbb{F}_n[\lambda]$. We will not give the specific polynomial basis for this matrix since in Theorems~\ref{CSonGeneralizadas} and~\ref{CharacterizationLength1or2}  we will consider broader families of matrices (that includes~(\ref{Barnet}) as a particular case) for which  such a polynomial basis exists. And in the proof of Lemma~\ref{GCMinBP(i-iii)}  we will see how the  polynomial basis is constructed. These kind of results  have motivated us to generalize the concept of companion matrix\footnote{Garnett et al. in~\cite{Garnett-Shader-Shader-VanDerDiessche}  employ the term  generalized companion matrix for those  matrices of order $n$ with $2n-1$ nonzero entries, $n-1$ are ones and $n$ are elements of the ring $\mathbb{F}[x_1,\ldots,x_n]$, with characteristic polynomial equal to~(\ref{PolynomialOfCompanion}). In the literature there exist other generalizations of companion matrices (see for instance the introduction of the recent work of De Ter\'an and Hernando~\cite{DeTeran-Hernando}).}.

A  matrix $A$ of order $n$ is said to be   \emph{companion  with respect to a polynomial basis}, \emph{PB-companion} by short, as long as:  $(i)$ $A$ has  $n^2 - n$  entries that are constants of the field $\mathbb{F}$; $(ii)$   the $n$ remaining entries of $A$ are the variables $x_1,\ldots,x_n$; and $(iii)$ the characteristic polynomial of $A$  is 
\begin{align*} 
\det(\lambda I_n-A)=p_n(\lambda) - x_1 p_{n-1}(\lambda)-\cdots- x_{n-1} p_1(\lambda) - x_n p_0(\lambda)
\end{align*}
where $\{p_n(\lambda), \ldots, p_1(\lambda), p_0(\lambda) \}$ is a polynomial basis of   $\mathbb{F}_n[\lambda]$. 

A companion matrix is  PB-companion  with respect to  the monomial basis $\{\lambda^n,\ldots,\lambda,1\}$  of $\mathbb{F}_n[\lambda]$.

The main concern  of this work is to describe both,  the companion unit lower Hessenberg matrices and the PB-companion unit lower Hessenberg  matrices. 

\section{Different sets of ULH matrices to be considered} \label{SomeNotation}

In the ULH matrices  that are candidates to be companion, the positions $(i_1,j_1),\ldots,(i_n,j_n)$ of the variables  $x_1,\ldots,x_n$ below the subdiagonal play  an essential role. In order to  avoid redundancies  we  define a total order in the set $\mathbb{Z}\times \mathbb{Z}$  by
$$
(i_r, j_r) \prec (i_s, j_s) \text{ if and only if } \begin{cases}
i_r-j_r < i_s-j_s  \\
\text{or}\\
i_r-j_r = i_s-j_s \text{ and } i_r < i_s.
\end{cases}
$$ 
This establishes the order in which $x_1,\ldots,x_n$ should be placed. 

%\begin{equation} \label{PositionXs}
% \text{$x_1,\ldots,x_n$ are the  $(i_1,j_1),\ldots,(i_n,j_n)$ entries with $(i_1, j_1) \prec \cdots \prec (i_n, j_n)$.} 
%\end{equation}
For the proper development of this work we  will introduce several sets of ULH matrices:
\begin{itemize}
 \item Let $\mathcal{H}$ be the set of ULH matrices of any order $n$  that have  $2n-1$ nonzero entries: $n-1$ entries in the superdiagonal equal to one, and $n$ entries below the superdiagonal equal to $x_1,\ldots,x_n$ that are placed  according to the order $\prec$. Let us describe the matrices of $ \mathcal{H}$ more precisely:
    \begin{eqnarray*} \label{DefinitionH}
    &H_{(i_1,j_1),\ldots,(i_n,j_n)} \text{ is the ULH matrix of order $n$ with $2n-1$ nonzero entries} \\
    &\text{whose $(i_1,j_1),\ldots,(i_n,j_n)$ entries are $x_1,\ldots,x_n$}\nonumber \\
    &\text{with $(i_1, j_1) \prec \cdots \prec (i_n, j_n)$ and  $1\leq j_k \leq i_k  \leq n$ for $k=1,\ldots,n$.} \nonumber 
    \end{eqnarray*}
    
%%%%%%%%%%%%%%%%%%%%%%%%%%%%%%%%%%%%%%%%%%%%%%%%%%%%%%%%%%%%    
%%%%%%%%%%%%%%%%%%%%%%%%%%%%%%%%%%%%%%%%%%%%%%%%%%%%%%%%%%%%

    \item Let $\mathcal{D}$ be the set composed by those matrices of $\mathcal{H}$ in which  the variable $x_1$ is located in the diagonal, $x_2$ in the subdiagonal, $x_3$ in the second subdiagonal, and so forth. Let us describe the matrices of $ \mathcal{D}$ more precisely:
    \begin{equation*} \label{DefinitionD}
    \text{$D_{i_1,i_2,\ldots,i_n}=H_{(i_1,i_1),(i_2,i_2-1)\ldots,(i_n,i_n-n+1)}$  whenever  $k\leq i_k$ for $k=1,\ldots,n$.}
    \end{equation*}

    \item Let $\mathcal{C}$ be the set composed by those matrices of $\mathcal{D}$ in which the variables $x_1,\ldots,x_n$ lay in the rectangular submatrix with top-right vertex the  $(i_1,i_1)$ entry   and with bottom-left vertex  the  $(n,1)$ entry.  Let us describe the matrices of $\mathcal{C}$ more precisely:
    \begin{equation} \label{DefinitionC}
    \text{$C_{i_1,\ldots,i_n}=D_{i_1,\ldots,i_n}$  whenever $i_k-k+1 \leq i_1 \leq  i_k$ for $k=1,\ldots,n$.}
    \end{equation}

    \item Let $\mathcal{G}$ be the set composed by those matrices of $\mathcal{C}$  which   have at least one variable on each one of the rows $i_1,\ldots,n$ and  on each one of the columns $1,\ldots,i_1$.  Let us describe the matrices of $\mathcal{G}$ more precisely:
    \begin{align*} \label{DefinitionG}
    \text{$G_{i_1,\ldots,i_n}=C_{i_1,\ldots,i_n}$  whenever } & \text{$\{i_1,\ldots,i_n\}=\{i_1,\ldots,n\}$ and}  \\
    & \text{$\{i_1,i_2-1,\ldots,i_n-n+1\}=\{1,\ldots,i_1\}$.} \nonumber
    \end{align*}

    \item And finally,  let $\mathcal{F}$ be the set composed by those matrices of $\mathcal{G}$ such that for   $k=2,\ldots,n$ the  variable $x_{k}$ is in the same row or in the same column than the variable $x_{k-1}$. So in a matrix of  $\mathcal{F}$ the variables  $x_1,\ldots,x_n$  form a lattice path starting on the  diagonal with $x_1$ and ending on the bottom-left corner with $x_n$. Let us describe the matrices of $\mathcal{F}$ more precisely:
    \begin{equation*} \label{DefinitionF}
    \text{$F_{i_1,\ldots,i_n}=G_{i_1,\ldots,i_n}$  whenever either $i_k=i_{k-1}$ or $i_k=i_{k-1}+1$ for $k=2,\ldots,n$.}
    \end{equation*}

%%%%%%%%%%%%%%%%%%%%%%%%%%%%%%%%%%%%%%%%%%%%%%%%%%%%%%%%%%%%%%%%%%%%%%%%
%%%%%%%%%%%%%%%%%%%%%%%%%%%%%%%%%%%%%%%%%%%%%%%%%%%%%%%%%%%%%%%%%%%%%%%%

\end{itemize}

As the  following  examples shows, we have   $\mathcal{H}\supsetneq  \mathcal{D}\supsetneq  \mathcal{C}\supsetneq  \mathcal{G}\supsetneq  \mathcal{F}.$   
$$ \begin{array}{|c|c|c|}   \hline
H_{(1,1),(3,3),(2,1),(5,4),(5,2)} \in \mathcal{H}\setminus \mathcal{D} &  D_{3,5,3,5,5} \in \mathcal{D}\setminus \mathcal{C} & C_{3,4,3,4,5} \in \mathcal{C} \setminus \mathcal{G}  \\ \hline \vspace{-3mm}
& &  \\  
  \begin{bmatrix}
x_1 & 1 & 0 & 0 & 0   \\
x_3 & 0 & 1 & 0 & 0   \\
0 & 0 & x_2 & 1 & 0   \\
0 & 0 & 0 & 0 & 1   \\
0 & x_5 & 0 & x_4 & 0   \\
\end{bmatrix} &
  \begin{bmatrix}
0 & 1 & 0 & 0 & 0   \\
0 & 0 & 1 & 0 & 0   \\
x_3 & 0 & x_1 & 1 & 0   \\
0 & 0 & 0 & 0 & 1   \\
x_5 & x_4 & 0 & x_2 & 0   \\
\end{bmatrix} &  \begin{bmatrix}
0 & 1 & 0 & 0 & 0   \\
0 & 0 & 1 & 0 & 0   \\
x_3 & 0 & x_1 & 1 & 0   \\
x_4 & 0 & x_2 & 0 & 1   \\
x_5 & 0 & 0 & 0 & 0   \\
\end{bmatrix}  \\ \hline
\end{array}
$$
$$ \begin{array}{|c|c|}  \hline
G_{3,4,3,5,5} \in \mathcal{G}\setminus \mathcal{F} & F_{3,4,4,4,5}\in \mathcal{F}  \\ \hline \vspace{-3mm}
&   \\  
  \begin{bmatrix}
0 & 1 & 0 & 0 & 0   \\
0 & 0 & 1 & 0 & 0   \\
x_3 & 0 & x_1 & 1 & 0   \\
0 & 0 & x_2 & 0 & 1   \\
x_5 & x_4 & 0 & 0 & 0   \\
\end{bmatrix} & \begin{bmatrix}
0 & 1 & 0 & 0 & 0  \\
0 & 0 & 1 & 0 & 0   \\
0 & 0 & x_1 & 1 & 0   \\
x_4 & x_3 & x_2 & 0 & 1   \\
x_5 & 0 & 0 & 0 & 0   \\
\end{bmatrix}  \\ \hline
\end{array}
$$

\medskip

{\bf Observations:} 1. The choice of letters for the sets of matrices was made according to the following mnemonics rule: $\mathcal{H}$ for Hessenberg (this is the largest set to consider), $\mathcal{D}$ for diagonal (each variable is in a different diagonal), $\mathcal{C}$ for companion (these are the sparse companion ULH matrices), $\mathcal{G}$ for gather (rows $i_1$ to $n$ and columns 1 to $i_1$ gather all variables with each row  and each column having at least one), and $\mathcal{F}$ for Fiedler (these are the Fiedler companion matrices).  

2. Each matrix $A$ of order $n$ has associated  a digraph of $n$ vertices in which the vertex $i$ is joined with the vertex $j$ by an oriented edge if and only if the $(i,j)$ entry of $A$ is nonzero.  If $A_1\in \mathcal{H}\setminus \mathcal{D}$, $A_2\in \mathcal{D}\setminus \mathcal{C}$, $A_3\in \mathcal{C}\setminus \mathcal{G}$, $A_4\in \mathcal{G}\setminus \mathcal{F}$, and $A_5\in \mathcal{F}$, then it is not difficult to see that the digraph associated to $A_i$ is not isomorphic to the digraph associated to $A_j$ for $1\leq i < j \leq 5$. In~\cite{Eastman-Kim-Shader-VanderMeulen}  the authors studied the digraphs associated to matrices of $\mathcal{C}$. 

\section{Companion ULH matrices} \label{ulHCompanionMatrices}

From now on we will frequently use the following definitions and notations:
\begin{enumerate}[a)]
    \item Given a  matrix $A$, if nothing is mentioned,  it will be assumed that it is square of order $n$. 
    \item\label{NotationCharPol} Instead of $\det(\lambda I_n-A)$ we will use  the shorter notation $P_A(\lambda)$ to refer to the characteristic polynomial of $A$.  
    \item\label{Not-emptyMatrix} The matrix $A$ of order 0 is the \emph{empty matrix},  with $\det(A)=1$ and   $P_A(\lambda)=1$.
    \item\label{Not-A0} The \emph{constant part} of a matrix $A$ whose entries are  elements of the ring $\mathbb{F}[x_1,\ldots,x_n]$ is the constant matrix   obtained from $A$ by doing $x_1=\cdots=x_n=0$.
    \item For $k=0,1,\ldots,n-1$ the \emph{$k$-subdiagonal} of $A$  is $(a_{k+1,1},a_{k+2,2},\ldots,a_{n,n-k})$. 
    \item $s_k(A)$ is the sum of the entries on the $k$-subdiagonal of $A$.
    \item $A[i_1,\ldots,i_r;j_1,\ldots,j_s]$ is the submatrix of all entries  in rows $i_1,\ldots,i_r$ and columns  $j_1,\ldots,j_s$.
    \item The \emph{$k$-block} of $A$ is the submatrix $A[k,\ldots,n;1,\ldots,k]$ of $A$ with top-right corner at position $(k,k)$ and bottom-left corner at position $(n,1)$.  
    \item  $A[k_1,\ldots,k_r]$ is $A[k_1,\ldots,k_r;k_1,\ldots,k_r]$. 
    \item The \emph{$k$-leading principal submatrix of $A$} is $A[1,\ldots,k]$.
    \item The \emph{$k$-trailing principal submatrix of $A$} is $A[n-k+1,\ldots,n]$.
    \item $A[i, \ldots,j]$  is  the empty matrix if $i,\ldots,j$ is not a valid range for $A$.
    \item A \emph{companion ULH matrix} is a companion matrix which is also ULH.

    \item If starting with a matrix $A$ we obtain $B$ by changing   zero entries of $A$ by nonzero constants, we will say that $B$ is a \emph{superpattern} of $A$. This terminology comes from~\cite{Eastman-Kim-Shader-VanderMeulen}.

    \item If starting with a ULH matrix $A$ we obtain  $B$ by changing   zero entries of $A$ below the superdiagonal  by nonzero constants, we will say that $B$ is a \emph{ULH superpattern} of $A$.
    
    \item   The set of all ULH superpatterns of a ULH matrix $A$ is denoted by $\widetilde{A}$. 
    
    \item The set of all  ULH superpatterns of all  matrices of,  respectively,  $\mathcal{H}$,  $\mathcal{D}$, $\mathcal{C}$,  $\mathcal{G}$, and  $\mathcal{F}$ is denoted by,  respectively, $\widetilde{\mathcal{H}}$,  $\widetilde{\mathcal{D}}$, $\widetilde{\mathcal{C}}$,  $\widetilde{\mathcal{G}}$, and  $\widetilde{\mathcal{F}}$.
    \end{enumerate}

\subsection{From sparse to non-sparse}

Eastman et al. characterized the sparse companion matrices.

\begin{theorem}~\cite{Eastman-Kim-Shader-VanderMeulen2} \label{CharSparseCompanion}
$A$ is a sparse companion matrix if and only if $A$ can be obtained from some matrix of $\mathcal{C}$ by any combination of transposition, permutation similarity and diagonal similarity.
\end{theorem}

And they also determined which are the sparse companion matrices within  $\mathcal{H}$. Note that the set of sparse companion matrices of  $\mathcal{H}$ is   the same as the set of sparse companion ULH matrices. So  we adapt their result to the notation we have introduced.

\begin{theorem}~\cite[Theorem 4.1]{Eastman-Kim-Shader-VanderMeulen}    \label{HMinusCNoCompanion}
The  sparse companion ULH matrices  are the matrices of  $\mathcal{C}$.
\end{theorem}

When we broaden our gaze to non-sparse companion matrices we will no longer obtain a result like  Theorem~\ref{CharSparseCompanion}. This is so because Deaett et al.~\cite{Deaett-Fischer-Garnett-VanderMeulen} showed one non-sparse companion matrix  that  can not be transformed into a ULH matrix by any combination of transposition, permutation similarity and diagonal similarity. So we will focus our efforts to obtain a result like  Theorem~\ref{HMinusCNoCompanion} also valid for \mbox{non-sparse} companion ULH matrices.  The following theorem tells us where we should look for the companion ULH matrices and how to characterize them.

\begin{theorem} \label{CCompanion}
Any   companion ULH matrix  belongs to $\widetilde{\mathcal{C}}$. Moreover, if $A\in \widetilde{C}_{i_1,\ldots,i_n}$ then $A$ is a  companion matrix if and only if the following conditions are met:
\begin{enumerate}[$(i)$]
    \item The constant part of $A$ is a nilpotent matrix;
    \item $A[1,\ldots,i_k-k]$ is a constant nilpotent matrix or the empty matrix for  $k=1,\ldots,n$;
    \item  $A[i_k+1,\ldots,n]$  is a constant nilpotent matrix or the empty matrix for  $k=1,\ldots,n$.    
\end{enumerate}  
If the only nonzero entries of the $i_1$-block of $A$  are $x_1,\ldots,x_n$ then condition $(i)$ can  be omitted.
\end{theorem}

\begin{proof}
The set of companion ULH matrices  is the same as the set of companion matrices of  $\widetilde{\mathcal{H}}$, so in order to prove that any companion ULH matrix  belongs to $\widetilde{\mathcal{C}}$ it is enough to demonstrate that any matrix of $\widetilde{\mathcal{H}} \setminus \widetilde{\mathcal{C}}$ is not companion.

If $B\in \widetilde{\mathcal{H}} \setminus \widetilde{\mathcal{D}}$ then there are two variables  $x_r$ and $x_s$ on the same $k$-subdiagonal of $B$.  In the characteristic polynomial of $B$  appear the terms $x_r \lambda^{n-k-1}$ and  $x_s\lambda^{n-k-1}$. So $B$ is not   companion.

If $B\in\widetilde{\mathcal{D}} \setminus \widetilde{\mathcal{C}}$ then $x_1$ is the $(i_1,i_1)$ entry of $B$ and there exists some $x_s$  out of the $i_1$-block of $B$. As $x_s$ is  on the $(s-1)$-subdiagonal of $B$ then in the characteristic polynomial of $B$  appears the term $x_1x_s\lambda^{n-1-s}$. So $B$ is not companion.

Now we go to the second part of the theorem. Recall the notation for the characteristic polynomial of a matrix that we introduced in~\ref{NotationCharPol}) at the beginning of Section~\ref{ulHCompanionMatrices}. If $A$ is a matrix of $\widetilde{C}_{i_1,\ldots,i_n}$ then all the terms of   $P_A(\lambda)$ in which  variables $x_1,\ldots,x_n$ do not appear conform  $P_{A_0}(\lambda)$ where $A_0$ is the constant part of $A$.  On the other hand, as $A$ is a ULH matrix and   $x_k$  is the $(i_k,i_k-k+1)$ entry of $A$ then $x_k$ appears in $P_A(\lambda)$ only in the term 
$$
-x_k \, P_{A[1,\ldots,i_k-k]}(\lambda) \, P_{A[i_k+1,\ldots,n]}(\lambda).
$$
Note that  $A[1,\ldots,i_k-k]$  is the empty matrix if $i_k=k$ and $A[i_k+1,\ldots,n]$ is the empty matrix if $i_k=n$. If we add the fact that  the variables $x_1,\ldots,x_n$ are in the $i_1$-block of $A$ then $A[1,\ldots,i_k-k]$ is a submatrix of the constant matrix $A[1,\ldots,i_1-1]$ and $A[i_k+1,\ldots,n]$ is a submatrix of the constant matrix $A[i_1+1,\ldots,n]$. Therefore 
\begin{eqnarray} \label{PolCharInPolBasis}
P_{A}(\lambda) &=& P_{A_0}(\lambda)-
\sum_{k=1}^n x_k \, P_{A[1,\ldots,i_k-k]}(\lambda) \, P_{A[i_k+1,\ldots,n]}(\lambda).  
\end{eqnarray}
Note that if some $P_{A[1,\ldots,i_k-k]}(\lambda) \, P_{A[i_k+1,\ldots,n]}(\lambda)$ contained a variable, then~(\ref{PolCharInPolBasis}) would not be true.

So $A$ is a companion matrix  if and only if equations~(\ref{PolCharInPolBasis}) and~(\ref{PolynomialOfCompanion}) match, that is,   if and only if 
\begin{align*}
\lambda^n&=P_{A_0}(\lambda); \\
\lambda^{n-k}&=P_{A[1,\ldots,i_k-k]}(\lambda)\, P_{A[i_k+1,\ldots,n]}(\lambda) \ \text{ for } k=1,\ldots,n; 
\end{align*}
if and only if 
\begin{align*}
\lambda^n&=P_{A_0}(\lambda); \\
\lambda^{i_k-k}&=P_{A[1,\ldots,i_k-k]}(\lambda) \ \text{ for } k=1,\ldots,n; \\
\lambda^{n-i_k}&=P_{A[i_k+1,\ldots,n]}(\lambda) \ \text{ for } k=1,\ldots,n; 
\end{align*}
if and only if 
\begin{align*}
& A_0 \ \text{ is nilpotent}; \\
& A[1,\ldots,i_k-k] \ \text{ is a constant nilpotent matrix or the empty matrix for } k=1,\ldots,n; 
\\
& A[i_k+1,\ldots,n] \ \text{ is a constant nilpotent matrix or the empty matrix for } k=1,\ldots,n.
\end{align*}

Let us prove the second part. If the only nonzero entries of the $i_1$-block of $A$  are $x_1,\ldots,x_n$ then  
\begin{align*} %\label{A00}
A_0=\left[\begin{array}{cccc}
 A[1,\ldots,i_1-1]  &  {e}_{i_1-1}  &  0 \\ 
 0&  0&    e_1^T \\ 
 0 &  0 &   A[i_1+1,\ldots,n] 
\end{array}\right].    
\end{align*}
So, if  $A$ meets conditions $(ii)$ and $(iii)$ then  $A_0$ is nilpotent in all three possible cases: 
\begin{enumerate}[$(a)$]
    \item If $1<i_1<n$ then  $A[1,\ldots,i_1-1]$ and $A[i_1+1,\ldots,n] $ are nilpotent. So it is $A_0$. 
    \item If $i_1=1$ then $A_0=\begin{smat}
0 & e_1^T \\ 0 & A[2,\ldots, n]
\end{smat}$ and $A[2,\ldots, n]$ is nilpotent. So it is $A_0$.
    \item If $i_1=n$ then $A_0=\begin{smat}
A[1,\ldots, n-1] & e_{n-1} \\ 0 & 0
\end{smat}$ and $A[1,\ldots, n-1]$ is nilpotent. So it is $A_0$.
\end{enumerate}
\end{proof}

\subsection{Nested nilpotent ULH matrices}

Let  $n_1,\ldots,n_s$ be integers with  $0\leq n_1,\ldots,n_s\leq n$.  We will denote by $\mathfrak{N}^n_L(n_1,\ldots,n_s)$ the set of constant ULH matrices of order $n$ such that  for each $h=1,\ldots,s$ with $n_h\neq  0$ the $n_h$-leading principal submatrix is nilpotent. And we will denote by $\mathfrak{N}^n_T(n_1,\ldots,n_s)$ the set of ULH matrices of order $n$ such that  for each $h=1,\ldots,s$ with $n_h\neq  0$ the $n_h$-trailing principal submatrix is nilpotent. By convenience we have included the possibility that the sequence $n_1,\ldots,n_s$ contains repeated elements or elements equal to zero. 

If $J_n=\begin{smat}
\vspace{-2mm} 0 & & 1 \\
 & \reflectbox{$\ddots$} & \\
1 & &  0
\end{smat}$ is the  \emph{backward identity matrix}    then
$$\mathfrak{N}^n_T(n_1,\ldots,n_s)=\big(
J_n\, \mathfrak{N}^n_L(n_1,\ldots,n_s)\, J_n
\big)^T.$$
The transposition is necessary to transform ULH matrices into ULH matrices. An example:
\begin{align*}
A &= \left[\begin{array}{rrrr}
 0 &\z1& 0 & \z 0 \\ \cline{1-1}
 3 &-5& 1 &  \z0 \\
 15&-28& 5 &  \z 1 \\  \cline{1-3}
 0 &0& 0 & 0 
\end{array}\right]\in \mathfrak{N}^4_L(1,3,4)\; \text{  and  }\;
\big(
J_n\, A \, J_n
\big)^T = \left[\begin{array}{rrrr}
 0 & 1 & 0 & 0 \\  \cline{2-4}
 0 & \z \mbox{$5$} & 1 & 0 \\
0 & \z \mbox{$-28$} & -5 & 1 \\ \cline{4-4}
 0 & \z \mbox{$15$} & 3 &  \z 0 
\end{array}\right]\in \mathfrak{N}^4_T(1,3,4).
\end{align*}

Now we will state Theorem~\ref{CCompanion} in terms of nested nilpotent matrices.

\begin{theorem}\label{CompanionInC}
Any   companion ULH matrix  belongs to $\widetilde{\mathcal{C}}$. 
Moreover, if $A \in \widetilde{C}_{i_1,\ldots,i_n}$ then $A$ is a  companion matrix if and only if the following conditions are met:
\begin{enumerate}[$(i)$]
    \item  The constant part of $A$  is a nilpotent matrix;
    \item $A[1,\ldots,i_1-1]$ belongs to $\mathfrak{N}^{i_1-1}_L(i_1-1,i_2-2,\ldots,i_n-n)$;
    \item  $A[i_1+1,\ldots,n]$ belongs to $\mathfrak{N}^{n-i_1}_T(n-i_1,n-i_2,\ldots,n-i_n)$.
\end{enumerate}  
If the only nonzero entries of the $i_1$-block of $A$  are $x_1,\ldots,x_n$ then condition $(i)$ can  be omitted.
\end{theorem}

\begin{proof} It is enough to prove that conditions (ii) and (iii) in Theorems~\ref{CCompanion}~and~\ref{CompanionInC} are equal. Taking into account~(\ref{DefinitionC}) we conclude that $0\leq i_k-k\leq i_1-1$ for each $k=1,\ldots,n$ and so conditions (ii) in both theorems are equal. Again, taking into account~(\ref{DefinitionC}) we conclude that $i_1\leq i_k\leq n$ for each $k=1,\ldots,n$ and so conditions (iii) in both theorems are equal. 
\end{proof}

\medskip

In Section~\ref{DescripcionCompanions} we will show how to parametrize   the sets  $\mathfrak{N}^n_L(n_1,\ldots,n_s)$ and  $\mathfrak{N}^n_T(n_1,\ldots,n_s)$. Previously we analyze a specially simple case.  

\begin{lemma} \label{MLHNilpotentAnidadas1a1}
$\mathfrak{N}^n_L(1,\ldots,n)=\mathfrak{N}^n_T(1,\ldots,n)=\{U_n\}$, where $U_n$ is the upper shift matrix with ones on the superdiagonal and zeroes elsewhere. 
\end{lemma}

\begin{proof}
We will prove by induction that $\mathfrak{N}^n_L(1,\ldots,n)=\{U_n\}$. For $n=1$, $\mathfrak{N}^1_L(1)=\{[\, 0 \,]\}=\{U_1\}$. Assume that it is true for $k$, that is, $\mathfrak{N}^k_L(1,\ldots,k)=\{U_k\}$. Let $A$ be a matrix of $\mathfrak{N}^{k+1}_L(1,\ldots,k+1)$. As $A[1,\ldots,k]\in \mathfrak{N}^{k}_L(1,\ldots,k)$ then $A[1,\ldots,k]=U_{k}$. As  $A$ is a ULH matrix then 
$$
A=\left[\begin{array}{cccccc}
& & & \multicolumn{1}{|r} 0 \quad & \\
& U_{k} & & \multicolumn{1}{|r} \vdots \quad \\
&  & & \multicolumn{1}{|r} 0 \quad \\ 
& & & \multicolumn{1}{|r} 1 \quad \\ \cline{1-3}
\, a_{k+1,1} & \cdots & a_{k+1,k} & a_{k+1,k+1} \
\end{array}\right].
$$
As $A$ is nilpotent then $a_{k+1,1}=\cdots=a_{k+1,k+1}=0$. So  $\mathfrak{N}^{k+1}_L(1,\ldots,k+1)=\{U_{k+1}\}$.

On the other hand, $$\mathfrak{N}^n_T(1,\ldots,n)=(J_n \, \mathfrak{N}^n_L(n_1,\ldots,n_r) \, J_n)^T=\{(J_n \,  U_n  \, J_n)^T\}=\{U_n\}.$$ 
\end{proof}

\subsection{Description  of the companions of $\widetilde{\mathcal{G}}$}

Even though Theorems~\ref{CCompanion} and ~\ref{CompanionInC} characterize  the companion matrices of $\widetilde{\mathcal{C}}$, we can give a nice and much easier to check description of the companion matrices of the subset $\widetilde{\mathcal{G}}$ of $\widetilde{\mathcal{C}}$. 

First we define the sets $\widehat{\mathcal{C}}$, $\widehat{\mathcal{G}}$, and $\widehat{\mathcal{F}}$  composed by those matrices  of, respectively, $\widetilde{\mathcal{C}}$, $\widetilde{\mathcal{G}}$, and $\widetilde{\mathcal{F}}$ such that all its nonzero entries are either in the superdiagonal (and are equal to 1)  or in the $i_1$-block determined by the position of the variable $x_1$. Note that $\widehat{\mathcal{C}}\supsetneq \widehat{\mathcal{G}} \supsetneq \widehat{\mathcal{F}}$. 

This  structure makes that if $A$ is a matrix of $\widehat{\mathcal{C}}$ then     
\begin{equation}\label{CharPolyCB}
  P_{A}(\lambda)=\lambda^n - s_0(A) \lambda^{n-1}-\cdots -s_{n-2}(A)\lambda-s_{n-1}(A)
\end{equation}
where $s_k(A)$ is the sum of the entries on the $k$-subdiagonal of $A$. Matching~(\ref{CharPolyCB}) with~(\ref{PolynomialOfCompanion}) we obtain the following characterization of the companion matrices of $\widehat{\mathcal{C}}$.

\begin{lemma}  \label{CompanionCB}
$A\in \widehat{\mathcal{C}}$ is companion  if and only if $s_{k-1}(A)=x_{k}$ for   $k=1,\ldots,n$.
\end{lemma} 

Let us see how to use this result to characterize the companion matrices of  $\widetilde{\mathcal{G}}$. 

\begin{theorem} \label{CompanionsInG}
$A\in \widetilde{\mathcal{G}}$ is  companion  if and only if $A\in  \widehat{\mathcal{G}}$ and $s_{k-1}(A)=x_{k}$ for  $k=1,\ldots,n$.
\end{theorem}
\begin{proof}
The sufficiency follows from   Lemma~\ref{CompanionCB}. 

The necessity. By Lemma~\ref{CompanionCB}, it is enough to prove that $A\in  \widehat{\mathcal{G}}$. Note that $A\in \widetilde{G}_{i_1,\ldots,i_n}$ for some $i_1,\ldots,i_n$. Then there exists at least one variable on each of the rows $i_1,i_1+1,\ldots,n$ and one variable on each of the columns $1,\ldots,i_1$.  From Theorem~\ref{CompanionInC} and Lemma~\ref{MLHNilpotentAnidadas1a1}  it follows that  
\begin{align*}
A[1,\ldots,i_1-1]&\in\mathfrak{N}^{i_1-1}_T(1,\ldots,i_1-1)=\{U_{i_1-1}\}; \text{ and } \\ A[i_1+1,\ldots,n]&\in\mathfrak{N}^{n-i_1}_T(1,\ldots,n-i_1)=\{U_{n-i_1}\}. 
\end{align*}
So  $A$ meets the conditions to be a matrix  of $\widehat{\mathcal{G}}$. 
\end{proof}

The interest of Theorem~\ref{CompanionsInG} is that we are able to parametrize all the ULH superpatterns of a considerable number of matrices of $\mathcal{C}$. Now we give an example of such a parametrization.

\begin{example}
Consider the matrix ${G}_{4,4,6,5,5,7,7}\in \mathcal{G}$. According to  Theorem~\ref{CompanionsInG} the parameterization  of the set of all the  companion matrices  within $\widetilde{G}_{4,4,6,5,5,7,7}$ is
$$
\left[\begin{array}{cccccccccccccc}
 0   & 1 &  0  & 0 &  0 & 0 & 0 \\
 0   & 0 &  1  & 0 &  0 & 0 & 0 \\
 0   & 0 &  0  & 1 &  0 & 0 & 0 \\  
 b   & a &  x_2  & x_1 & 1 & 0 & 0 \\ 
 x_5 & x_4 & 1-a   & 0 & 0 & 1 & 0 \\
 0   & d & c   & x_3 & 0 & 0 & 1  \\
 x_7 & x_6 & 1-d   & 1-b-c & 0 & 0 & 0 \\
\end{array}\right].$$
\end{example}

A consequence of Theorem~\ref{CompanionsInG} is related with an open question  stated by Deaett et al.~\cite{Deaett-Fischer-Garnett-VanderMeulen}:  ``\emph{We wonder if, in producing a companion matrix by changing some zero entries of a Fiedler companion matrix $F_{i_1,\ldots,i_n}$ by nonzero constants, the extra nonzero entries are always restricted to the submatrix corresponding to the $i_1$-block}". They partially confirmed that supposition.

\begin{theorem} \label{FiedlerConjecture}  \cite[Theorem 5.4]{Deaett-Fischer-Garnett-VanderMeulen}
Let $A$ be a matrix  obtained from the Fiedler companion matrix $F_{i_1,\ldots,i_n}$ by changing  zero entries  that are not in the  $i_1$-block. Then $A$ is not companion. 
\end{theorem}

We contribute solving the case in which  the  change of zero entries in the  Fiedler companion matrix is  below the superdiagonal. Indeed, this is particular case of  Theorem~\ref{CompanionsInG}: \emph{``$A\in \widetilde{\mathcal{F}}$ is companion  if and only if $A\in  \widehat{\mathcal{F}}$ and $s_{k-1}(A)=x_{k}$ for  $k=1,\ldots,n$"}. We write this in the same language as Theorem~\ref{FiedlerConjecture}.

\begin{theorem} \label{SupperpFCoincBelowDiag}
Let $A$ be a matrix  obtained from the Fiedler companion matrix $F_{i_1,\ldots,i_n}$ by changing  zero entries  that are below the superdiagonal. Then $A$ is companion  if and only if $A$ is obtained by only changing  zero entries that are in the  $i_1$-block and $s_{k-1}(A)=x_{k}$ for  $k=1,\ldots,n$.
\end{theorem}

It remains unknown if exists a companion matrix which is obtained from some Fiedler companion matrix $F_{i_1,\ldots,i_n}$ by  changing at least one zero entry of  the  $i_1$-block and at least one zero entry above the superdiagonal. We have computational evidence that this is not the case for  $n\leq 7$.
 
\subsection{Parameterization  of the companions of $\widetilde{\mathcal{C}}\setminus \widetilde{\mathcal{G}}$} \label{DescripcionCompanions}

When we have a matrix $C_{i_1,\ldots,i_n}\in \mathcal{C}\setminus \mathcal{G}$ the parametrization of the set of all the companion matrices within $\widetilde{C}_{i_1,\ldots,i_n}$ is not so easy as with matrices of $\mathcal{G}$. We will show it  with an specific example. Let $$
C_{7,7,9,7,8,9,10,8,9,10}=\left[\begin{array}{cccccccccccccc}
0 & 1 & 0 & 0   & 0 &  0  & \multicolumn{1}{|r}0 &  \multicolumn{1}{|r}0 & 0 & 0  \\
0 & 0 & 1 & 0   & 0 &  0  & \multicolumn{1}{|r}0 &  \multicolumn{1}{|r}0 & 0 & 0  \\
0 & 0 & 0 & 1   & 0 &  0  & \multicolumn{1}{|r}0 &  \multicolumn{1}{|r}0 & 0 & 0  \\
0 & 0 & 0 & 0   & 1 &  0  & \multicolumn{1}{|r}0 &  \multicolumn{1}{|r}0 & 0 & 0 \\
0 & 0 & 0 & 0   & 0 &  1  & \multicolumn{1}{|r}0 &  \multicolumn{1}{|r}0 & 0 & 0 \\
0 & 0 & 0 & 0   & 0 &  0  & \multicolumn{1}{|r}1 &  \multicolumn{1}{|r}0 & 0 & 0 \\  \hline
0 & 0 & 0 & x_4 & 0 & x_2 & x_1 & \multicolumn{1}{|r}1 & 0 & 0   \\ \cline{8-10}
x_8 & 0 & 0 & x_5   & 0 & 0   & 0 & \multicolumn{1}{|r}0 & 1 & 0 \\
x_9 & 0 & 0 & x_6   & 0 & 0   & x_3 & \multicolumn{1}{|r}0 & 0 & 1  \\
x_{10} & 0 & 0 & x_7   & 0 & 0   & 0 & \multicolumn{1}{|r}0 & 0 & 0 \\
\end{array}\right]=
\left[\begin{array}{cccccc}
& U_6 & & \multicolumn{1}{|r} \rm{$e_6$} & & \multicolumn{1}{|r} 0 \\ \hline
& & & & & \multicolumn{1}{|r}  \rm{$e^T_1$} \\ \cline{6-6}
& & X & & & \multicolumn{1}{|r}  \rm{$U_3$} 
\end{array}\right].
$$
The set $\widetilde{C}_{7,7,9,7,8,9,10,8,9,10}$ of its  ULH superpatterns is composed by the matrices 
\begin{equation}\label{ulHSuperpatternGeneral}
C=\left[\begin{array}{cccccc}
& A & & \multicolumn{1}{|r} \rm{$e_6$} & & \multicolumn{1}{|r} 0 \\ \hline
& & & & & \multicolumn{1}{|r}  \rm{$e^T_1$} \\ \cline{6-6}
& & Y & & & \multicolumn{1}{|r}  \rm{$B$} 
\end{array}\right]
\end{equation}
where  $A$ is a  ULH superpattern  of $U_6$, $B$ is a ULH superpattern  of $U_3$,  and $Y$ is a superpattern of $X$. By Theorem~\ref{CompanionInC}, $C$ is  companion  if and only if the following conditions are met:
\begin{enumerate}[$(i)$]
    \item The constant part of $C$ is a  nilpotent matrix;
    \item $A$ belongs to $\mathfrak{N}^6_L(i_1-1,i_2-2,\ldots,i_{10}-10)=\mathfrak{N}^6_L(6,5,6,3,3,3,3,0,0,0)=\mathfrak{N}^6_L(3,5,6)$;
    \item $B\in\mathfrak{N}^3_T(10-i_1,10-i_2,\ldots,10-i_{10})=\mathfrak{N}^3_T(3,3,1,3,2,1,0,2,1,0)=\mathfrak{N}^3_T(1,2,3).$
\end{enumerate} 
We wish to parameterize the companions matrices of type~(\ref{ulHSuperpatternGeneral}). We  divide it in three cases:  
\begin{enumerate}[(1)]
\item $A=U_6$ and $B=U_3$. So we consider the companion matrices of type 
\begin{equation*}
C_1=\left[\begin{array}{cccccc}
& U_6 & & \multicolumn{1}{|r} \rm{$e_6$} & & \multicolumn{1}{|r} 0 \\ \hline
& & & & & \multicolumn{1}{|r}  \rm{$e^T_1$} \\ \cline{6-6}
& & Y & & & \multicolumn{1}{|r}  \rm{$U_3$} 
\end{array}\right]. 
\end{equation*}
where $Y$ is a superpattern of $X$. The solution is given in Lemma~\ref{CompanionCB}: $C_1$ is a companion matrix if and only if $s_{k-1}(C_1)=x_{k}$ for each $k=1,\ldots,n$. So 
$$Y=\begin{bmatrix} 
h      & f & d & x_4 & a & x_2 & x_1 \\
x_8      & i & g & x_5 & b & -a   & 0  \\
x_9    & j & -h-i  & x_6 & e & c   & x_3  \\
x_{10} & 0 & -j & x_7 & -f-g & -d-e   & -b-c \\
\end{bmatrix}  \ \text{ where } a,b,\ldots,j\in \mathbb{F}.
$$

\item $Y=X$. So we consider the companion matrices of  type 
$$
C_2=\left[\begin{array}{cccccc}
& A & & \multicolumn{1}{|r} \rm{$e_6$} & & \multicolumn{1}{|r} 0 \\ \hline
& & & & & \multicolumn{1}{|r}  \rm{$e^T_1$} \\ \cline{6-6}
& & X & & & \multicolumn{1}{|r}  \rm{$B$} 
\end{array}\right]. 
$$
where  $A$ is a constant ULH matrix of order 6, and $B$ is a constant ULH matrix of order 3. By   Theorem~\ref{CompanionInC}, $C_2$ is a companion matrix if and only if  $A\in\mathfrak{N}^6_L(3,5,6)$ and  $B\in\mathfrak{N}^3_T(1,2,3)$.

By  Lemma~\ref{MLHNilpotentAnidadas1a1},  $\mathfrak{N}^3_T(1,2,3)=\{U_3\}$. So $B=U_3$.

On the other hand, as any matrix of  $\mathfrak{N}^6_L(3,5,6)$ is a ULH matrix then 
\begin{equation*} 
A=\left[\begin{array}{cccccc}
a_{11} & 1 & 0 & \multicolumn{1}{|r}0 & 0 & \multicolumn{1}{|r}0  \\ 
a_{21} & a_{22} & 1 & \multicolumn{1}{|r}0 & 0 & \multicolumn{1}{|r}0  \\ 
a_{31} & a_{32} & a_{33} & \multicolumn{1}{|r}1 & 0 & \multicolumn{1}{|r}0  \\ \cline{1-3}
a_{41} & a_{42} & a_{43} & a_{44} & 1 & \multicolumn{1}{|r}0 \\
a_{51} & a_{52} & a_{53} & a_{54} & a_{55} & \multicolumn{1}{|r}1  \\ \cline{1-5}
a_{61} & a_{62} & a_{63} & a_{64} & a_{65} & a_{66}
\end{array}\right]
\end{equation*}
where $A[1,2,3]$, $A[1,2,3,4,5]$ and  $A[1,2,3,4,5,6]$ are nilpotent.  The ULH submatrices $A[1]$ and $A[1,2]$ have no restriction. So, assume that we have assigned arbitrary constant values to $a_{11},a_{21},a_{22}$.  As  $A[1,2,3]$ is a nilpotent ULH matrix, then $a_{31},a_{32},a_{33}$ should take values that make  the  characteristic polynomial of $A[1,2,3]$ to be $\lambda^3$.  Corollary~\ref{CSonGeneralizadas}, that we will state later, implies that these values are unique and can be obtained in a practical way since   $(1,a_{33},a_{32},a_{31})$ are the coordinates of $\lambda^3$ with respect to certain  basis of $\mathbb{F}_3[\lambda]$. Let us go to $A[1,2,3,4]$, for which we have no further restrictions. So, assume that we have assigned arbitrary constant values to $a_{41},a_{42},a_{43},a_{44}$. As $A[1,2,3,4,5]$ is a nilpotent ULH matrix,  then we repeat the arguments above in order to obtain the adequate values for $a_{51},a_{52},a_{53},a_{54},a_{55}$. And, finally, $A[1,2,3,4,5,6]$ is also a nilpotent ULH matrix which we get  by doing $a_{61}=\cdots=a_{66}=0$ since $A[1,2,3,4,5]$ is nilpotent and has  characteristic polynomial $\lambda^5$. Moreover, by Corollary~\ref{CSonGeneralizadas}, this solution is unique.  So, doing the corresponding calculations, we obtain
$$
A=\small\left[
\begin{array}{cccccccccc}
 a & 1 & 0 & 0 & 0 & 0  \\
 c & b & 1 & 0 & 0 & 0  \\
 -a^3-2 a c-b c & -a^2-a b-b^2-c & -a-b & 1 & 0 & 0  \\
 g & f & e & d & 1 & 0  \\
 a_{51}  & a_{52} & a e+b e-d e-f & -d^2-e & -d & 1  \\
 0 & 0 & 0 & 0 & 0 & 0  \\
\end{array}
\right] \ \text{ with } a,b,\ldots,g\in \mathbb{F}
$$
where $a_{51}=a^3 e+2 a c e-a g+b c e-c f-d g$ and $a_{52}=a^2 e+a b e+b^2 e-b f+c e-d f-g$.

\item $Y\neq X$ and   $(A,B)\neq (U_6,U_3)$.  It is possible to give a partial parametrization for this case, although we do not have the complete parametrization.  In  Example 4.5 of~\cite{Deaett-Fischer-Garnett-VanderMeulen} the authors gave, for $C_{3,4,3,4,7,6,7}$, a partial parametrization for the analogous  case.
\end{enumerate}

\section{The PB-companion ULH matrices} \label{NonStandardPolynomialBasis}

Recall that  a  matrix $A$ of order $n$ is said to be  \emph{PB-companion} or  \emph{companion  with respect to a polynomial basis}, as long as:  $(i)$ $A$ has  $n^2 - n$  entries that are constants of a field $\mathbb{F}$; $(ii)$   the $n$ remaining entries of $A$ are the variables $x_1,\ldots,x_n$; and $(iii)$ the characteristic polynomial of $A$  is 
\begin{align*} %\label{gencommatcoef}
\det(\lambda I_n-A)=p_n(\lambda) - x_1 p_{n-1}(\lambda)-\cdots- x_{n-1} p_1(\lambda) - x_n p_0(\lambda)
\end{align*}
where $\{p_n(\lambda), \ldots, p_1(\lambda), p_0(\lambda) \}$ is a polynomial basis of   $\mathbb{F}_n[\lambda]$. 

The issue we are concerned about in this section is the characterization of the PB-companion ULH matrices or, what is the same, the PB-companion matrices within the set $\widetilde{\mathcal{H}}$. We will see that the set of the PB-companion matrices of  $\widetilde{\mathcal{H}}$ is greater than $\widetilde{\mathcal{C}}$ and strictly smaller that $\widetilde{\mathcal{H}}$, and that the intermediate set $\widetilde{\mathcal{D}}$ is not useful since as  $\widetilde{\mathcal{D}}\setminus\widetilde{\mathcal{C}}$ does not contain any  PB-companion matrix. So we need a new  set between $\widetilde{\mathcal{H}}$ and $\widetilde{\mathcal{C}}$ in which we will locate the PB-companion ULH matrices.

Let $\mathcal{B}$ be the set composed by those matrices of $\mathcal{H}$ in which the variables $x_1,\ldots,x_n$ are on the $t$-block for a certain $t\in\{1,\ldots,n\}$ (there may be more than one such $t$). Note that   $\mathcal{B}\cap \mathcal{D} = \mathcal{C}$. Let us describe the matrices of $ \mathcal{B}$ more precisely: 
\begin{equation*} %\label{DefinitionB}
\text{$B_{(i_1,j_1),\ldots,(i_n,j_n)}=H_{(i_1,j_1),\ldots,(i_n,j_n)}$ whenever    $1\leq  j_1,\ldots,j_n \leq i_1,\ldots,i_n\leq n$.}
\end{equation*}
The variables   $x_1, \ldots, x_n$ are on the   $t$-block of $B_{(i_1,j_1),\ldots,(i_n,j_n)}$   if and only if $$\max\{j_1, \ldots, j_n\}\leq t \leq \min\{i_1, \ldots, i_n\}.$$
As the  following  examples shows, we have      $\mathcal{H}\supsetneq \mathcal{B} \supsetneq \mathcal{C}$.
$$ \begin{array}{|c|c|} \hline  
H_{(1,1),(3,3),(2,1),(5,4),(5,2)} \in \mathcal{H}\setminus \mathcal{B} &  B_{(3,1),(4,2),(4,1),(5,2),(5,1)} \in \mathcal{B}\setminus \mathcal{C}  \\ \hline \vspace{-3mm}
&  \\  
  \begin{bmatrix}
x_1 & 1 & 0 & 0 & 0   \\
x_3 & 0 & 1 & 0 & 0   \\
0 & 0 & x_2 & 1 & 0   \\
0 & 0 & 0 & 0 & 1   \\
0 & x_5 & 0 & x_4 & 0   \\
\end{bmatrix} &
  \left[\begin{array}{ccccc}
0 &1 & 0 & 0 & 0\\\cline{1-2}
0 & 0 & \multicolumn{1}{|r}1 & 0 & 0\\ \cline{1-3}
x_1 & 0 & \multicolumn{1}{|r}0 & \multicolumn{1}{|r}1 & 0\\
x_3 & x_2 & \multicolumn{1}{|r}0 & \multicolumn{1}{|r}0 & 1\\
x_5 & x_4 & \multicolumn{1}{|r}0 & \multicolumn{1}{|r}0 & 0 \\
\end{array}\right]   \\ \hline
\end{array}
$$
Let us see  that $\widetilde{\mathcal{B}}$ is the right place where we should look for  PB-companion ULH matrices.

\begin{theorem} \label{HMinusCNoSonGeneralizadas}
Any PB-companion ULH matrix  belongs to $\widetilde{\mathcal{B}}$.
\end{theorem}
\begin{proof}
The problem of determining the PB-companion ULH matrices  is the same that the problem of determining the PB-companion matrices within  $\widetilde{\mathcal{H}}$.

Let   $A\in\widetilde{H}_{(i_1,j_1),\ldots,(i_n,j_n)}$ be a matrix of $\widetilde{\mathcal{H}} \setminus \widetilde{\mathcal{B}}$. As $A$ is not in  $\widetilde{\mathcal{B}}$  then 
$$\max\{j_1, \ldots, j_n\}=j_r>i_s=\min\{i_1, \ldots, i_n\}.$$
for some $r,s\in\{1,\ldots,n\}$. Therefore  $A[j_s,\ldots,i_s]$ contains $x_{s}$, $A[j_r,\ldots,i_r]$ contains $x_{r}$, and both matrices are disjoint. So in $P_A(\lambda)$ the term $x_{s}x_{r}\lambda^{n-(i_s-j_s+1)-(i_r-j_r+1)}$ appears, and this term can not cancel out. We conclude that  $A$ is not a companion matrix. 
\end{proof}

In the proof of the previous result we have seen that in the characteristic polynomial of a matrix of $\widetilde{\mathcal{H}} \setminus \widetilde{\mathcal{B}}$ inevitably the product of two variables $x_r x_s$ appear. For matrices of $\widetilde{\mathcal{B}}$ no product of two variables appear, as we will show in the next Lemma. Moreover, the polynomials $p_{n-k}(\lambda)$ that accompany each variable $x_k$ are crucial for knowing when a matrix of $\widetilde{\mathcal{B}}$ is PB-companion. We will show some useful properties of those polynomials.

\begin{lemma} \label{GCMinBP(i-iii)}
Let   $A$ be a matrix of $\widetilde{B}_{(i_1,j_1),\ldots,(i_n,j_n)}$. Then the following  is true: 
\begin{enumerate}[(i)]
\item  \label{MonicPolynomialsForB}
The characteristic polynomial of $A$ can be written as
\begin{align}\label{CharPolGenCompMat(i-iii)} 
P_A(\lambda)=p_n(\lambda)-x_1 p_{n-1}(\lambda)-\cdots-x_n p_0(\lambda)
\end{align}
where $p_n(\lambda),  p_{n-1}(\lambda),\ldots ,  p_0(\lambda)$ are monic polynomials of $\mathbb{F}_n[\lambda]$. 
    \item \label{GradosPolinomios}  
    $\deg(p_n(\lambda))=n$  and   $\deg(p_{n-k}(\lambda))=n-i_k+j_k-1$ for  $k=1,\ldots,n.$
    \item \label{OrderDegrees}  $n=\deg(p_n(\lambda)) \geq \deg(p_{n-1}(\lambda))\geq \ldots \geq \deg(p_0(\lambda)) $. 
    \item \label{DegreeOfThePs} If $A$ is   PB-companion then for   $k=0,1,\ldots,n$ the degree of  $p_{n-k}(\lambda)$ is at least $n-k$.
    \item \label{VariableKOverDiagonalK}  If $A$ is   PB-companion then for   $k=1,\ldots,n$ the variable $x_{k}$ is above the $k$-subdiagonal.
    
\end{enumerate}
\end{lemma}

\begin{proof} \begin{enumerate}[(i)]
\item 
To show that the characteristic polynomial of $A$ can be written as in~(\ref{CharPolGenCompMat(i-iii)}) we use the same arguments that in the proof of Theorem~\ref{CCompanion}. This leads us to conclude that
    \begin{eqnarray} \label{PolCharInPolBasis2}
    P_{A}(\lambda) &=& P_{A_0}(\lambda)- \sum_{k=1}^n x_k \,P_{A[1,\ldots,j_k-1]}(\lambda) \, P_{A[i_k+1,\ldots,n]}(\lambda),  
    \end{eqnarray}
    where $A_0$ is the constant part of $A$. Since $x_1,\ldots,x_n$ are in some $t$-block of $A$ then for $k=1,\ldots,n$ the matrix $A[1,\ldots,j_k-1]$ is  empty  or a submatrix of the constant matrix $A[1,\ldots,t-1]$, and the matrix $A[i_k+1,\ldots,n]$ is empty or a submatrix of the  constant matrix $A[t+1,\ldots,n]$. Since the variables $x_1,\ldots, x_n$ do not appear in either $P_{A[1,\ldots,j_k-1]}(\lambda)$ or   $P_{A[i_k+1,\ldots,n]}(\lambda)$, then we can    
    match~(\ref{PolCharInPolBasis2}) and~(\ref{CharPolGenCompMat(i-iii)}) to conclude that  $p_n(\lambda)=P_{A_0}(\lambda)$ and $p_{n-k}(\lambda)=P_{A[1,\ldots,j_k-1]}(\lambda) \, P_{A[i_k+1,\ldots,n]}(\lambda)$ for $k=1,\ldots,n$ are monic polynomials of $\mathbb{F}_n[\lambda]$.

    \item  $\deg(p_n(\lambda))=\deg(P_{A_0}(\lambda))=n$ and  for $k=1,\ldots,n$ $$\deg(p_{n-k}(\lambda))=\deg(P_{A[1,\ldots,j_k-1]}(\lambda))  \deg(P_{A[i_k+1,\ldots,n]}(\lambda))=j_k-1+n-i_k.$$

\item It follows from~(\ref{GradosPolinomios}) by taking into account  that  the variables $x_1,\ldots,x_n$ are placed in $A$ according to the order $\prec$, which implies that $i_1-j_1\leq i_2-j_2\leq \cdots \leq i_n-j_n$. 

\item If  $A$ is PB-companion then $\{p_0(\lambda),p_1(\lambda),\ldots,p_n(\lambda)\}$ is a basis of  $\mathbb{F}_n[\lambda]$. Therefore the inequality $\deg(p_{n-k}(\lambda))\geq n-k$ for each $k=1,\ldots,n$ follows from~(\ref{OrderDegrees}).

\item   If $x_k$ is on the $h$-subdiagonal then, by~(\ref{GradosPolinomios}) and~(\ref{DegreeOfThePs}),  $\deg(p_{n-k}(\lambda))=n-h-1\geq n-k$.  So  $h< k$.
\end{enumerate}
\end{proof}

Lemma~\ref{GCMinBP(i-iii)} help us to find a first great set of PB-companion ULH matrices.   

\begin{theorem} \label{CSonGeneralizadas}
Each  matrix of $\widetilde{\mathcal{C}}$ is  PB-companion. 
\end{theorem}

\begin{proof}
If $A\in \widetilde{\mathcal{C}}_{(i_1,j_1),\ldots,(i_n,j_n)}$ then $j_k=i_k-k+1$ for  $k=1,\ldots,n$. From  Lemma~\ref{GCMinBP(i-iii)}~(\ref{MonicPolynomialsForB})-(\ref{GradosPolinomios}) it follows that $P_A(\lambda)=p_n(\lambda) - x_1 p_{n-1}(\lambda)-\cdots - x_n p_0(\lambda)$ with $\deg(p_n(\lambda))=n$  and   
$$\deg(p_{n-k}(\lambda))=n-i_k+j_k-1=n-i_k+i_k-k+1-1=n-k \ \text{ for } k=1,\ldots,n.$$
So $\{p_n(\lambda), \ldots, p_1(\lambda), p_0(\lambda) \}$ is a  basis of   $\mathbb{F}_n[\lambda]$ and $A$ is PB-companion. 
\end{proof}

\subsection{A criterion for PB-companion matrices}

Lemma~\ref{GCMinBP(i-iii)}~(\ref{MonicPolynomialsForB}) permits us to introduce a constant matrix associated to each matrix of $\widetilde{\mathcal{B}}$.

\begin{definition} \label{MA}
Let   $A$ be a matrix of $\widetilde{\mathcal{B}}$   whose characteristic polynomial is
\begin{align*}%\label{CharPolGenCompMat} 
P_A(\lambda)=p_n(\lambda)-x_1 p_{n-1}(\lambda)-\cdots-x_{n-1}p_1(\lambda)-x_n p_0(\lambda).
\end{align*}
We will denote by    $\mathfrak{M}_A$ the constant matrix $[p_{ij}]_{i,j=0}^n$ of order $n+1$ such that the entries on its rows are taken from   $p_i(\lambda)=p_{i0}+p_{i1}\lambda+\cdots+p_{in}\lambda^n$ for $i=0,1,\ldots,n$. 
\end{definition}

The main objective of Section~\ref{NonStandardPolynomialBasis} is the characterization of the PB-companion ULH matrices. Relying on Theorem~\ref{HMinusCNoSonGeneralizadas}  we will achieve this objective  in the next result    with  the characterization of the PB-companion matrices of $\widetilde{\mathcal{B}}$. 

\begin{theorem} \label{pbCompanionMA}
Any PB-companion ULH matrix  belongs to $\widetilde{\mathcal{B}}$. Moreover, if $A\in \widetilde{\mathcal{B}}$ then $A$ is PB-companion if and only if $\mathfrak{M}_{A}$ is nonsingular.
\end{theorem} 

\begin{proof}
The first sentence is Theorem~\ref{HMinusCNoSonGeneralizadas}.  Now, let $A$ be a matrix of $\widetilde{\mathcal{B}}$. Then   $A$ is PB-companion if and only if $\{p_0(\lambda),p_1(\lambda),\ldots,p_n(\lambda)\}$ is a basis of $\mathbb{F}_n[\lambda]$ if and only if $\mathfrak{M}_{A}$ is nonsingular. 
\end{proof}

\begin{example} We  wish to determine if it is PB-companion the matrix 
$$
A  =\left[\begin{array}{ccccccc}
2 & 1 & 0 & 0 & 0 & 0 & 0 \\
0 & 3 & 1 & 0 & 0 & 0 & 0 \\
1 & 3 & -2 & 1 & 0 & 0 & 0 \\
4 & 0 & 2 & -1 & 1 & 0 & 0 \\ \cline{1-5}
x_5 & 1 & -3 & x_2 & x_1 & \multicolumn{1}{|r}1 & 0 \\
x_7 & x_6 & 2 & 5 & x_3 & \multicolumn{1}{|r}3 & 1 \\
1 & 0 & 4 & x_4 & 1 & \multicolumn{1}{|r}4 & 2
\end{array}\right] \in \widetilde{\mathcal{B}}_{(5,5), (5,4), (6,5), (7,4), (5,1), (6,2), (6,1)}.
$$
The characteristic polynomial of $A$ is  
\begin{align*} 
P_A(\lambda)=p_7(\lambda)-x_1 p_{6}(\lambda)-\cdots-x_{6}p_1(\lambda)-x_7 p_0(\lambda)
\end{align*}
where  
$$\left.\begin{array}{rllll}
x_7 \rightarrow &p_0(\lambda)=P_{A[7]}(\lambda)=-2+\lambda; \\
x_6 \rightarrow&p_1(\lambda)=P_{A[1]}(\lambda)\, P_{A[7]}(\lambda)=4-4\lambda+\lambda^2; \\
x_5 \rightarrow&p_2(\lambda)=P_{A[6,7]}(\lambda)=2-5\lambda+\lambda^2; \\
x_4 \rightarrow&p_3(\lambda)=P_{A[1,2,3]}(\lambda)=17-7\lambda-3\lambda^2+\lambda^3; \\
x_3 \rightarrow&p_4(\lambda)=P_{A[1,2,3,4]}(\lambda)\, P_{A[7]}(\lambda)=-2-39\lambda+44\lambda^2-8\lambda^3-4\lambda^4+\lambda^5;  \\
x_2 \rightarrow&p_5(\lambda)=P_{A[1,2,3]}(\lambda)\, P_{A[6,7]}(\lambda)=34-99\lambda+46\lambda^2+10\lambda^3-8\lambda^4+\lambda^5; \\
x_1 \rightarrow&p_6(\lambda)=P_{A[1,2,3,4]}(\lambda)\, P_{A[6,7]}(\lambda)=2+35\lambda-123\lambda^2+76\lambda^3-7\lambda^5+\lambda^6; \\
1 \, \rightarrow&p_7(\lambda)=P_{A_0}(\lambda)=208-317\lambda+168\lambda^2-129\lambda^3+73\lambda^4-7\lambda^6+\lambda^7.
\end{array}\right.$$
The   matrix $\mathfrak{M}_{A}=[p_{ij}]_{i,j=0}^7$ is
\begin{align*}
\begin{array}{c|cccccccc|} 
 & \lambda^0 & \lambda^1 & \lambda^2 & \lambda^3 & \lambda^4 & \lambda^5 & \lambda^6 & \lambda^7  \\ \cline{1-9}
p_{0}(\lambda)   & \y -2 & \y 1 & \y 0 &   0 &   0 &   0  &   0 &   0  \\  %\cline{3-3}
p_{1}(\lambda)   & \y 4 & \y -4 & \y 1 &   0 &   0 &   0 &   0  &   0 \\  
p_{2}(\lambda)   & \y 2 & \y -5 & \y 1 &   0 &   0 &   0 &   0 &   0  \\  %\cline{4-6}
p_{3}(\lambda)   & 17 & -7 & -3 & \y 1 &   0  &   0  &   0 &   0 \\  %\cline{6-9}
p_{4}(\lambda)   & -2 & -39 & 44 & -8 & \y -4 & \y 1  &   0 &   0 \\  
p_{5}(\lambda)   & 34 & -99 & 46 & 10 & \y -8 & \y 1  &   0 &   0 \\    
p_{6}(\lambda)   & 2 & 35 & -123 & 76 &  0 &  -7  & \y 1 &   0  \\   
p_{7}(\lambda)   & 208 & -317 & 168 &  -129 &  73 & 0 &  -7  & \y 1\\    \cline{1-9}
\end{array}
\end{align*}
By Theorem~\ref{pbCompanionMA},  $A$ is PB-companion  since  $\mathfrak{M}_{A}[0,1,2]$, $\mathfrak{M}_{A}[3]$, $\mathfrak{M}_{A}[4,5]$,  $\mathfrak{M}_{A}[6]$, and $\mathfrak{M}_{A}[7]$ (the grey blocks in the diagonal) are nonsingular. 
\end{example}

In general, to determine if a given  matrix $A$ of $\widetilde{\mathcal{B}}$ is  PB-companion needs some cumbersome calculations to obtain the polynomials that go into the characteristic polynomial of $A$. Nevertheless, in some cases  the calculations are going to simplify so much that we will conclude easily if $A$ is or it is not  PB-companion. This is because we do not need to known all the entries on the matrix $\mathfrak{M}_{A}$ to conclude that it is nonsingular. We only need to known the submatrices that were marked in gray in the previous example.

\begin{theorem} \label{pbCompanionConcatenation}
Let   $A$ be a matrix of $\widetilde{\mathcal{B}}$  with 
$P_A(\lambda)=p_n(\lambda)-x_1 p_{n-1}(\lambda)-\cdots-x_n p_0(\lambda)$ and let the  concatenation associated to $A$ be 
    \begin{equation*} 
            (0,1,\ldots,n)=(b_1,\ldots,e_1)^\frown\ldots^\frown(b_s,\ldots,e_s),
    \end{equation*}
    where  $e_1, \ldots, e_{s}$ are  the  $t\in\{0,1,\ldots, n\}$ such  that the degree of $p_{t}(\lambda)$ is $t$.  
    
    Then $A$  is PB-companion if and only if $\mathfrak{M}_{A}[b_h,\ldots,e_h]$ is nonsingular for    $h=1,\ldots,s$.
\end{theorem}
\begin{proof}
 From Lemma~\ref{GCMinBP(i-iii)}~(\ref{OrderDegrees})  and~(\ref{DegreeOfThePs}), if $A$ is PB-companion then  $\mathfrak{M}_{A}$ is a  lower triangular block matrix with $\mathfrak{M}_{A}[b_1,\ldots,e_1],\ldots,\mathfrak{M}_{A}[b_s,\ldots,e_s]$ as the blocks on its  diagonal. From Theorem~\ref{pbCompanionMA},   $A$  is  PB-companion if and only if $\mathfrak{M}_{A}$ is nonsingular if and only if $\mathfrak{M}_{A}[b_h,\ldots,e_h]$ is nonsingular for  each  $h=1,\ldots,s$.
\end{proof}

\subsection{Concatenations with components of length at most two} \label{SomeSpecialCases}

If  $A\in \widetilde{\mathcal{B}}$ then $(0,1,\ldots,n)=(0)^\frown(1)^\frown\ldots^\frown(n)$ is the concatenation associated to $A$ if and only if $\deg(p_t(\lambda))=t$   for  $t=0,1,\ldots,n$ if and only if  $A$ belongs to $\widetilde{\mathcal{C}}$. So, the following is a restatement of Theorem~\ref{CSonGeneralizadas} in terms of the length of the components of the concatenation: \emph{Each matrix of $\widetilde{\mathcal{B}}$ whose associated concatenation  has only components of length one is PB-companion}. 

In the next result we  characterize the PB-companion matrices of $\widetilde{\mathcal{B}}$ whose associated concatenation has all its components of length at most two.  

\begin{theorem}\label{CharacterizationLength1or2}
Let   $A$ be a matrix of $\widetilde{\mathcal{B}}$  with 
$$P_A(\lambda)=p_n(\lambda)-x_1 p_{n-1}(\lambda)-\cdots-x_n p_0(\lambda).$$ 
We define the  \emph{concatenation associated to $A$} as 
\begin{equation} \label{Concatenation}
            (0,1,\ldots,n)=(b_1,\ldots,e_1)^\frown\ldots^\frown(b_s,\ldots,e_s)
\end{equation}
where  $e_1, \ldots, e_{s}$ are  the  $t\in\{0,1,\ldots, n\}$ such  that the degree of $p_{t}(\lambda)$ is $t$.  

Then $A$ is a PB-companion matrix such that the length of the components of the concatenation associated to $A$ is either one or two if and only if on each  $k$-subdiagonal of $A$ 
            \begin{enumerate}[I.]
               \item \label{1var} the only variable is $x_{k+1}$; or
               \item \label{2var} the only two variables are $x_{k+1}$ and $x_{k+2}$, and $\displaystyle\sum_{r=i-k}^{i} a_{rr} \neq \sum_{r= i'-k}^{i'} a_{rr}$ if  $x_{k+1}$ is the $(i,i-k)$ entry of $A$ and $x_{k+2}$ is the $(i',i'-k)$ entry of $A$; or
               \item \label{0var} no variable appears.
           \end{enumerate}
\end{theorem}

\begin{proof}
$\Rightarrow$ Assume that $A$ is a PB-companion matrix of $\widetilde{\mathcal{B}}$ such that the length of the components of the concatenation associated to $A$ is either one or two.   

We will  consider all the possibilities for each $k$-subdiagonal of $A$:
  \begin{enumerate}[1)]
  \item On the $k$-subdiagonal of $A$ there are no variables.  This is case~\ref{0var}.

    \item \label{2} On the $k$-subdiagonal the variable with smallest index is $x_{h}$ with $h\leq k$. Then $x_h$ is on or below the $h$-subdiagonal. A  contradiction with Lemma~\ref{GCMinBP(i-iii)}~(\ref{VariableKOverDiagonalK}).
    
    \item \label{3} On the $k$-subdiagonal the variable with smallest index is  $x_{h}$ with $h\geq  k+2$. Then neither $x_{k+1}$ is on the $k$-subdiagonal nor $x_{k+2}$ is on the $(k+1)$-subdiagonal. Equivalently,   $\deg(p_{n-k-2}(\lambda))\neq n-k-2$ and $\deg(p_{n-k-1}(\lambda))\neq n-k-1$. So the  concatenation has a component of length at least three. A contradiction with the hypothesis.

   \item \label{firstxk+1} On the $k$-subdiagonal the variable with smallest index is  $x_{k+1}$. We consider three subcases: 
   \begin{enumerate}[i)]
       \item The  only variable on the $k$-subdiagonal  is $x_{k+1}$. This is case~\ref{1var}.
       \item \label{k-subXk+1Xk+2} The  only variables on the $k$-subdiagonal  are   $x_{k+1}$ and  $x_{k+2}$. Three new possibilities appear:
       \begin{enumerate}[a)]
           \item The variable $x_{k+3}$ is on the $(k+1)$-subdiagonal.  The same situation as in~\ref{3}). 
           \item The variable $x_{k+3}$ is on the $(k+h)$-subdiagonal with $h\geq 3$. The same situation as in ~\ref{2}). 
           \item \label{Xk+3}  The variable $x_{k+3}$ is on the $(k+2)$-subdiagonal. So we have that          
           $$
           \deg(p_{n-k-3}(\lambda))= n-k-3, \, \deg(p_{n-k-2}(\lambda))\neq n-k-2, \text{ and } \deg(p_{n-k-1}(\lambda))= n-k-1.
           $$
           Therefore $\,^\frown(n-k-2,n-k-1)^\frown$ is a component of the concatenation of length two. As $x_{k+1}$ is on the $k$-subdiagonal then $x_{k+1}$ is the $(i,i-k)$ entry of $A$ for some $i$ and 
            \begin{align*}
                p_{n-k-1}(\lambda)&=P_{A[1,\ldots,i-k-1]}(\lambda) \, P_{A[i+1,\ldots,n]}(\lambda)  \\ 
                &=(\lambda^{i-k-1}-\sum_{r=1}^{i-k-1}a_{rr}\lambda^{i-k-2}+\cdots)\,(\lambda^{n-i}-\sum_{r=i+1}^n a_{rr}\lambda^{n-i-1}+\cdots)\\
                &=\lambda^{n-k-1}-\sum_{r\in\{1,\ldots,n\}\setminus\{i-k,\ldots,i\}}a_{rr}\lambda^{n-k-2}+\cdots.
            \end{align*}
            Analogously, as $x_{k+2}$ is on the $k$-subdiagonal then $x_{k+2}$ is the $(i',i'-k)$ entry of $A$ for some $i'$ and             
            \begin{align*}
                p_{n-k-2}(\lambda)&=\lambda^{n-k-1}-\sum_{r\in\{1,\ldots,n\}\setminus\{i'-k,\ldots,i'\}}a_{rr}\lambda^{n-k-2}+\cdots.
            \end{align*}          
            Therefore 
           $$\mathfrak{M}_{A}[n-k-2,n-k-1]=\begin{bmatrix}
           p_{n-k-2,n-k-2} & p_{n-k-2,n-k-1} \\ 
           p_{n-k-1,n-k-2} & p_{n-k-1,n-k-1}
           \end{bmatrix}=\begin{bmatrix}
           -\sum_{r\in\{1,\ldots,n\}\setminus\{i-k,\ldots,i\}}a_{rr} & 1 \\ 
           -\sum_{r\in\{1,\ldots,n\}\setminus\{i'-k,\ldots,i'\}}a_{rr} & 1
           \end{bmatrix}.$$
           By Theorem~\ref{pbCompanionConcatenation}, $\mathfrak{M}_{A}[n-k-2,n-k-1]$ is nonsingular. And this is so  if and only if $\displaystyle\sum_{r=i-k}^i a_{rr}\neq \sum_{r=i'-k}^{i'} a_{rr}$. This is case~\ref{2var}.                
           \end{enumerate}
       
       \item On the $k$-subdiagonal  there are, at least, three variables   $x_{k+1}$, $x_{k+2}$ and $x_{k+3}$. In this case neither $x_{k+2}$ is on the $(k+1)$-subdiagonal of $A$ nor  $x_{k+3}$ is on the $(k+2)$-subdiagonal of $A$. Equivalently,  $\deg(p_{n-k-3}(\lambda))\neq n-k-3$ and $\deg(p_{n-k-2}(\lambda))\neq n-k-2$. So the  concatenation has a component of length at least three. A contradiction with the hypothesis.
   \end{enumerate}

\end{enumerate}
$\Leftarrow$ Assume that  each  $k$-subdiagonal verifies \ref{1var}, or \ref{2var}, or \ref{0var}. It is useful to note that  
\begin{equation}\label{xhH2H1}
\text{each variable $x_h$ is either on the $(h-1)$ or on the $(h-2)$-subdiagonal of $A$.}
\end{equation}

Let us  prove that the $k$-subdiagonal verifies \ref{2var} if and only if the $(k+1)$-subdiagonal verifies \ref{0var}. For the sufficiency assume that the $k$-subdiagonal verifies \ref{2var}, that is, that on the $k$-subdiagonal the only variables are $x_{k+1}$ and $x_{k+2}$. Then $x_{k+2}$ is not on   the  $(k+1)$-subdiagonal which, by discarding non-viable options, implies that the  $(k+1)$-subdiagonal  verifies \ref{0var}.  For the necessity, assume that the $(k+1)$-subdiagonal verifies \ref{0var}, that is, that the $(k+1)$-subdiagonal has no variables. Then $x_{k+2}$ is on the $k$-subdiagonal because of~(\ref{xhH2H1}), what makes the only viable option that $x_{k+1}$ is on the $k$-subdiagonal and that the $k$-subdiagonal verifies \ref{2var}.

By what we have just shown, if we traverse the subdiagonals in order we will  sometimes find consecutive subdiagonals that verify \ref{2var} and \ref{0var}, and the rest of subdiagonals must verify \ref{1var}. The translation into components of the concatenation~(\ref{Concatenation}) is that two consecutive subdiagonals that verify \ref{2var} and  \ref{0var} correspond to a component of length two, and a subdiagonal that verifies \ref{1var} corresponds to a component of length one. Let us see why this is so:

\begin{itemize}
    \item If two consecutive subdiagonals $k$ and $k+1$  verify \ref{2var} and \ref{0var}, then the only variables on the $k$-subdiagonal are $x_{k+1}$ and $x_{k+2}$, and the $(k+1)$-subdiagonal has no variables. Furthermore, $x_{k+3}$ must be in the $(k+2)$-subdiagonal because of~(\ref{xhH2H1}). So $$\deg(p_{n-k-3}(\lambda))= n-k-3, \ \deg(p_{n-k-2}(\lambda))\neq n-k-2,  \text{ and } \deg(p_{n-k-1}(\lambda))= n-k-1,$$ which means that $\,^\frown(n-k-2,n-k-1)^\frown$ is a component of the concatenation~(\ref{Concatenation}) of length two.  If  $x_{k+1}$ is on the $(i,i-k)$ entry of $A$ and $x_{k+2}$  is on the $(i',i'-k)$ entry of $A$ then (as we saw  above in  \ref{Xk+3})
    $$\mathfrak{M}_{A}[n-k-2,n-k-1]=\begin{bmatrix}
           -\sum_{r\in\{1,\ldots,n\}\setminus\{i-k,\ldots,i\}}a_{rr} & 1 \\ 
           -\sum_{r\in\{1,\ldots,n\}\setminus\{i'-k,\ldots,i'\}}a_{rr} & 1
           \end{bmatrix},$$
    which is nonsingular since  $\displaystyle\sum_{r=i-k}^{i} a_{rr} \neq \sum_{r= i'-k}^{i'} a_{rr}$ by hypothesis.  

    \item If the $k$-subdiagonal verifies \ref{1var}, then the only variable on the $k$-subdiagonal is $x_{k+1}$. Furthermore, $x_{k+2}$ must be in the $(k+1)$-subdiagonal because of~(\ref{xhH2H1}). So $$\deg(p_{n-k-2}(\lambda))= n-k-2 \text{ and }  \deg(p_{n-k-1}(\lambda))= n-k-1,$$ which means that  $\,^\frown(n-k-1)^\frown$ is a component of the concatenation~(\ref{Concatenation}) of length one. The corresponding principal submatrix of $\mathfrak{M}_{A}$ is $\mathfrak{M}_{A}[n-k-1]=[1]$, since  $p_0(\lambda), \ldots, p_n(\lambda)$ are monic by Theorem~\ref{GCMinBP(i-iii)}~(\ref{MonicPolynomialsForB}). 
\end{itemize}

Finally, by Theorem~\ref{pbCompanionConcatenation}, $A$ is PB-companion because we have shown that all the principal submatrices of $\mathfrak{M}_A$ that correspond to the components of the concatenation are nonsingular. 
\end{proof}

\begin{example}
The matrix of   $\widetilde{B}_{(5,5), (5,4), (6,5),  (5,2), (6,3), (6,1)}$  given by
$$A  =\begin{bmatrix}
a_{11} & 1 & 0 & 0 & 0 & 0  \\
a_{21} & a_{22} & 1 & 0 & 0 & 0  \\
a_{31} & a_{32} & a_{33} & 1 & 0 & 0  \\
a_{41} & a_{42} & a_{43} & a_{44} & 1 & 0  \\ 
a_{51} & x_4 & a_{53} & x_2 & x_1 &  1 \\
x_6 & a_{62} & x_5 & a_{63} & x_3 &  a_{66} \\
\end{bmatrix}.
$$
is, by Theorem~\ref{CharacterizationLength1or2},  PB-companion if and only if $a_{44}\neq a_{66}$ and $a_{22}\neq a_{66}$.

Let us see it in detail.  If the characteristic polynomial of $A$ is 
$$P_A(\lambda)=p_6(\lambda)-x_1 p_{5}(\lambda)-x_2 p_{4}(\lambda)-x_3 p_{3}(\lambda)-x_4 p_{2}(\lambda)-x_{5}p_1(\lambda)-x_6p_0(\lambda)$$
then  
$$\left.\begin{array}{rlllll}
x_6 \rightarrow&p_0(\lambda)=P_{A[\emptyset]}(\lambda)=1. \\
x_5 \rightarrow&p_1(\lambda)=P_{A[1,2]}(\lambda)=\lambda^2+(-a_{11}-a_{22})\lambda+\cdots; \\
x_4 \rightarrow&p_2(\lambda)=P_{A[1]}(\lambda)\, P_{A[6]}(\lambda)=\lambda^2+(-a_{11}-a_{66})\lambda+\cdots; \\
x_3 \rightarrow&p_3(\lambda)=P_{A[1,2,3,4]}(\lambda)=\lambda^4+(-a_{11}-a_{22}-a_{33}-a_{44})\lambda^3+\cdots; \\
x_2 \rightarrow&p_4(\lambda)=P_{A[1,2,3]}(\lambda)\, P_{A[6]}(\lambda)=\lambda^4+(-a_{11}-a_{22}-a_{33}-a_{66})\lambda^3+\cdots;  \\
x_1 \rightarrow&p_5(\lambda)=P_{A[1,2,3,4]}(\lambda)\, P_{A[6]}(\lambda)=\lambda^5+\cdots; \\
1 \, \rightarrow&p_6(\lambda)=P_{A_0}(\lambda)=\lambda^6+\cdots; \\
\end{array}\right.$$
And therefore the matrix $\mathfrak{M}_{A}$ of coefficients is
\begin{align*}
\begin{array}{c|ccccccc|} 
    & 1 & \lambda & \lambda^2 & \lambda^3 & \lambda^4 & \lambda^5 & \lambda^6 \\  \cline{1-8}
p_{0}(\lambda)   & \y 1 & 0 & 0 & 0 & 0 & 0 & 0 \\  %\cline{3-3}
p_{1}(\lambda)   & * & \y -a_{11}-a_{22} & \y 1 & 0 & 0 & 0 & 0 \\  
p_{2}(\lambda)   & * & \y -a_{11}-a_{66} & \y 1 & 0 & 0 & 0 & 0 \\  %\cline{4-6}
p_{3}(\lambda)   & * & * & * & \y -a_{11}-a_{22}-a_{33}-a_{44} & \y 1 & 0 & 0 \\  %\cline{6-9}
p_{4}(\lambda)   & * & * & * & \y -a_{11}-a_{22}-a_{33}-a_{66} & \y 1 &  0 &  0 \\  
p_{5}(\lambda)   & * & * & * & * &  * & \y 1 &  0 \\  
p_{6}(\lambda)   & * & * & * & * &  * &  * & \y 1 \\  \cline{1-8}
\end{array}
\end{align*}
Note that the concatenation associated to $A$ is 
\bigskip
\begin{align*}
(0,1,2,3,4,5,6) & = (\smash{\overset{\overset{b_1, e_1}{\shortparallel}}{0}})^\frown (\smash{\overset{\overset{b_2}{\shortparallel}}{1}}, \smash{\overset{\overset{e_2}{\shortparallel}}{2}})^\frown (\smash{\overset{\overset{b_3}{\shortparallel}}{3}}, \smash{\overset{\overset{e_3}{\shortparallel}}{4}})^\frown (\smash{\overset{\overset{b_4, e_4}{\shortparallel}}{5}})^\frown(\smash{\overset{\overset{b_5, e_5}{\shortparallel}}{6}}),
\end{align*}
where $0,2,4,5,6$ are all $t\in\{0,1,2,3,4,5,6\}$ such  that  $\deg(p_t(\lambda))=t$. So  $A$ is PB-companion if and only if $\det(\mathfrak{M}_{A}[1,2])=-a_{22}+a_{66}\neq 0$  and $\det(\mathfrak{M}_{A}[3,4])=-a_{44}+a_{66}\neq 0$.
\end{example}

\bibliographystyle{plain}

\begin{thebibliography}{1}

\bibitem{Barnett75}
S. Barnett,
\newblock A companion matrix analogue for orthogonal polynomials.
\newblock Linear Algebra and its Applications 12 (1975) 197--208.

\bibitem{Barnett81}
S. Barnett, 
\newblock Congenial matrices. 
\newblock Linear Algebra and its Applications 41 (1981), 277--298.

\bibitem{Deaett-Fischer-Garnett-VanderMeulen}
L. Deaett, J. Fischer, C. Garnett, K.N. VanderMeulen, 
\newblock Non-sparse Companion Matrices. 
\newblock Electronic Journal of Linear Algebra 35 (2019), 223--247.

\bibitem{DeTeran-Hernando}
F. De Ter\'{a}n, C. Hernando,
\newblock A note on generalized companion pencils in the monomial basis,
\newblock Revista de la Real Academia de Ciencias Exactas, F\'{i}sicas y Naturales. Serie A. Matem\'{a}ticas (2020) 114: 8. https://doi.org/10.1007/s13398-019-00760-y.

\bibitem{Eastman-Kim-Shader-VanderMeulen}
B. Eastman, I.J. Kim, B.L. Shader, K.N. Vander Meulen,
\newblock Companion matrix patterns,
\newblock  Linear Algebra and its Applications  463 (2014) 255--272.

\bibitem{Eastman-Kim-Shader-VanderMeulen2}
B. Eastman, I.J. Kim, B.L. Shader, K.N. Vander Meulen,
\newblock Corrigendum to "Companion matrix patterns" [Linear Algebra Appl. 463 (2014) 255--272],
\newblock  Linear Algebra and its Applications  538 (2018) 255--272.

\bibitem{Farahat-Ledermann}
H.K. Farahat, W. Ledermann
\newblock Matrices with prescribed characteristic polynomial,
\newblock Proc. Edinburgh Math. Soc. 11 (1958/1959) 143--146.

\bibitem{Fiedler} M. Fiedler, 
\newblock A note on companion matrices, 
\newblock Linear Algebra Appl. 372 (2003) 325--331.

\bibitem{Garnett-Shader-Shader-VanDerDiessche} 
C. Garnett, B.L. Shader, C.L. Shader, P. van den Driessche,
\newblock Characterization of a family of generalized companion matrices, 
\newblock Linear Algebra Appl. 498 (2016) 360--365.

\bibitem{Ma-Zhan} 
C. Ma, X. Zhan, 
\newblock Extremal sparsity of the companion matrix of a polynomial, 
\newblock Linear Algebra Appl. 438 (2013) 621--625.

\bibitem{Maroulas-Barnett} 
J. Maroulas, S. Barnett, 
\newblock Polynomials with respect to a general basis, Part I, Theory, 
\newblock Journal of Mathematical Analysis and Appl. 72 (1977) 177--194.

\end{thebibliography}

\end{document}